\documentclass[reqno]{amsart}

 \usepackage{microtype}
\usepackage{hyperref}
         
\usepackage{amsmath}             
\usepackage{amsfonts}             
\usepackage{amsthm}               
\usepackage{amsbsy}
\usepackage{amssymb}
\usepackage{amsthm}
\usepackage[initials]{amsrefs}

\usepackage{enumerate}

\newtheorem{thm}{Theorem}[section]
\newtheorem{lem}[thm]{Lemma}
\newtheorem{prop}[thm]{Proposition}
\newtheorem{cor}[thm]{Corollary}

\theoremstyle{definition}
\newtheorem{defn}[thm]{Definition}
\newtheorem{nota}[thm]{Notation}
\newtheorem{rem}[thm]{Remark}
\newtheorem{exam}[thm]{Example}

\newtheorem{method}[thm]{Method}

\usepackage{color}

\newcommand{\bC}{{\mathbb{C}}}

\newcommand{\bR}{{\mathbb{R}}}
\newcommand{\bD}{{\mathbb{D}}}
\newcommand{\bF}{{\mathbb{F}}}

\newcommand{\A}{{\mathcal{A}}}

\renewcommand{\H}{{\mathcal{H}}}

\newcommand{\M}{{\mathcal{M}}}

\renewcommand{\O}{{\mathcal{O}}}

\newcommand{\U}{{\mathcal{U}}}

\renewcommand{\phi}{\varphi}

\newcommand{\fA}{{\mathfrak{A}}}

\newcommand{\fM}{{\mathfrak{M}}}
\newcommand{\fN}{{\mathfrak{N}}}
\newcommand{\fR}{{\mathfrak{R}}}

\newcommand{\qand}{\quad\text{and}\quad}
\newcommand{\qqand}{\qquad\text{and}\qquad}

\newcommand{\proj}{\mathrm{Proj}}
\renewcommand{\Re}{\mathrm{Re}}
\renewcommand{\Im}{\mathrm{Im}}
\newcommand{\tr}{\mathrm{Tr}}
\newcommand{\diag}{\mathrm{diag}}
\newcommand{\sa}{\mathrm{sa}}
\newcommand{\conv}{\mathrm{conv}}

\begin{document}

\nocite{*}

\title{Numerical Ranges in II$_1$ Factors}

\author{Ken Dykema}
\author{Paul Skoufranis}
\address{Department of Mathematics, Texas A\&M University, College Station, Texas, USA, 77843-3368}
\email{kdykema@math.tamu.edu}
\email{pskoufra@math.tamu.edu}

\date{\today}
\thanks{The first author was supported in part by NSF grant DMS-1202660}

\begin{abstract}
In this paper, we generalize the notion of the $C$-numerical range of a matrix to operators in arbitrary
tracial von Neumann algebras.
For each self-adjoint operator $C$, the $C$-numerical range of such an operator is defined; it is a compact, convex subset of $\bC$.
We explicitly describe the $C$-numerical ranges of several operators and classes of operators.
\end{abstract}

\maketitle

\section{Introduction}

An interesting invariant of an operator is its numerical range.
Given a Hilbert space $\H$ and a bounded linear operator $T : \H \to \H$, the \emph{numerical range} of $T$ is the set of complex numbers
\[
W_1(T) = \{ \langle T\xi, \xi\rangle_\H \, \mid \, \xi \in \H, \left\|\xi\right\|_\H = 1\}.
\]

The Hausdorff-Toeplitz Theorem (see \cites{H1919, T1918}) states that the numerical range of an operator is always a convex subset.  Furthermore, when restricting to finite dimensional $\H$, the numerical range of a matrix is compact and can be used to obtain several interesting structural results, such as that a matrix of trace zero is always unitarily equivalent to a matrix with zeros along the diagonal.

The numerical range of a matrix is often substantially larger than the spectrum and yields cruder information about the matrix.  For example, if $N$ is a normal matrix, then $W_1(N)$ is the convex hull of the eigenvalues of $N$.  Therefore, precise information about the eigenvalues of $N$ cannot be obtained from $W_1(N)$.  

In \cite{H1967}, Paul Halmos proposed a generalization of the numerical range of a matrix.  For each $\xi \in \bC^n$ with $\left\| \xi\right\|_2 = 1$ and $T \in \M_n(\bC)$, we have
\[
\langle T\xi, \xi\rangle_{\bC^n} = \tr(TP_\xi)
\]
where $\tr$ is the (unnormalized) trace and $P_\xi \in \M_n(\bC)$ is the rank one projection onto $\bC \xi$.  Thus, for $T \in \M_n(\bC)$ and $k \in \{1,\ldots, n\}$, the \emph{$k$-numerical range of $T$} defined as
\[
W_k(T) =  \left\{\frac{1}{k} \tr(TP) \, \mid \, P \in \mathcal{M}_n(\mathbb{C}) \mbox{ a projection of rank }k\right\}.
\]

C. A. Berger showed, using the Hausdorff-Toeplitz Theorem and the fact that $W_1(T)$ is convex,
that each $W_k(T)$ is a convex set (see \cite{H1967}*{Solution 211}).
Operators' $k$-numerical ranges have been extensively studied and much is known.
For example \cite{FW1971}*{Theorem 1.2} shows
\[
W_k(T) = \frac{1}{k}\{ \tr(TX) \, \mid \, 0 \leq X \leq I_n, \tr(X) = k\}.
\]
It is clear that the set on the right-hand-side of the above equation is a convex set, yet this did not produce an new proof of Berger's result as \cite{FW1971}*{Theorem 1.2} relied on of Berger's result.
These $k$-numerical ranges provide substantially more information about a matrix than the numerical range alone.
Indeed, if $N\in \M_n(\bC)$ is a normal matrix with eigenvalues $\{\lambda_j\}^n_{j=1}$
listed according to their multiplicities, then, by \cite{FW1971}*{Theorem 1.5}, the $k$-numerical range of $N$ is the convex hull of the set
\begin{align*}
\left\{ \frac{1}{k}\sum_{j \in K} \lambda_j \, \mid \, J \subseteq \{1,\ldots, n\}, |J| = k \right\}.  
\end{align*}
By varying $k$, these sets provide enough information to determine the eigenvalues of $N$ and, thus, to determine
$N$ up to unitary equivalence.

In \cite{W1975}, Westwick analyzed a generalization of the $k$-numerical ranges of a matrix which was later further generalized by Golberg and Straus in \cite{GS1977}.  Given two matrices $C, T \in \M_n(\bC)$, the \emph{$C$-numerical range of $T$} is defined to be the set
\begin{align}
W_C(T) = \{ \tr (TU^*CU) \, \mid \, U \in \M_n(\bC) \text{ a unitary}\}.\label{eq:matrix-C-nr}
\end{align}
It is not difficult to see that if $C_k \in \M_n(\bC)$ is a matrix with $\frac{1}{k}$ along the diagonal precisely $k$ times and zeros elsewhere, then $W_{C_k}(T) = W_k(T)$.  Thus, the $C$-numerical ranges are indeed generalizations of the $k$-numerical ranges.

Using ideas from \cite{H1919}, Westwick in \cite{W1975} demonstrated that if $C \in \M_n(\bC)$ is self-adjoint, then $W_C(T)$ is a convex set.  However, Westwick also showed that if $C = \diag(0, 1, i) \in \M_3(\bC)$, then $W_C(C)$ is not convex.  Based on \cite{W1975} and \cite{GS1977}, in \cite{P1980} Poon gave another proof that the $C$-numerical ranges are convex for self-adjoint $C \in \M_n(\bC)$.  Poon's work gave an alternate description of the $C$-numerical range based on a notion of majorization for $n$-tuples of real numbers.
This notion of majorization is the one appearing in a classical theorem of Schur (\cite{S1923}) and Horn (\cite{H1954}) characterizing the possible diagonal $n$-tuples of a self-adjoint matrix based on its eigenvalues. 

As the notion of majorization has an analogue in arbitrary  tracial
von Neumann algebras, the goal of this paper is to examine $C$-numerical ranges in arbitrary von Neumann algebras.  In light of
the example of Westwick  given above, we will restrict our attention to self-adjoint $C$.  Furthermore, we note that analogues of the $k$-numerical ranges inside diffuse von Neumann algebras have been previously studied in \cites{AAW1999, AA2002, AA2003-1, AA2003-2}.  Consequently, the results contained in this paper are a mixture of generalizations of results from \cites{AAW1999, AA2002, AA2003-1, AA2003-2}, new proofs  of
results in \cites{AAW1999, AA2002, AA2003-1, AA2003-2}, and additional results.
This paper contains a total of six sections, including this one, and is structured as follows.

Section \ref{sec:Defns-and-Basic-Results} begins by recalling a notion of majorization for elements of $L^\infty[0,1]$.  The generalization of $C$-numerical ranges to tracial von Neumann algebras is then obtained by applying majorization to eigenvalue functions of self-adjoint operators.  After many basic properties of $C$-numerical ranges are demonstrated, several important results, such as the fact that $C$-numerical ranges are independent of the von Neumann algebra under consideration, are obtained.  Of importance are the results that $C$-numerical ranges are always compact, convex sets of $\bC$ and, if one restricts to type II$_1$ factors, one can define $C$-numerical ranges using the closed unitary orbit of $C$ instead of the notion of majorization.  In addition, we demonstrate the $C$-numerical range of $T$ is continuous in both $C$ and $T$, and  we
demonstrate results from \cites{AAW1999, AA2002, AA2003-1, AA2003-2}  that follow immediately from this different view.

Section \ref{sec:SA} is dedicated to describing the $C$-numerical ranges of self-adjoint operators via eigenvalue functions.  This is particularly important for Section \ref{sec:Computing} which demonstrates a method for computing $C$-numerical ranges of operators based on knowledge of $C$-numerical ranges of self-adjoint operators.  This is significant as numerical ranges of matrices are often difficult to compute (see \cite{KRS1997} for the $3 \times 3$ case). 

Section \ref{sec:Examples} computes $\alpha$-numerical ranges (i.e. the generalization of the $k$-numerical range of a matrix) for several operators. Although computing the $k$-numerical ranges of a matrix is generally a hard task, there are several interesting examples of operators in II$_1$ factors whose $\alpha$-numerical ranges can be explicitly described.  In particular, we demonstrate the existence of normal and non-normal operators whose $\alpha$-numerical ranges agree for all $\alpha$.

Section \ref{sec:Diagonals} concludes the paper by examining the relationship between $\alpha$-numerical ranges and conditional expectations of operators onto subalgebras.  In particular, we demonstrate that a scalar $\lambda$ is in the $\alpha$-numerical range of an operator $T$ in a II$_1$ factor if and only if there exists diffuse abelian von Neumann subalgebra $\A$ such that the trace of the spectral projection of the expectation of $T$ onto $\A$ corresponding to the set $\{\lambda\}$ is at least $\alpha$.

\section{Definitions and Basic Results}
\label{sec:Defns-and-Basic-Results}

In this section, we generalize the notion of the $C$-numerical range of a matrix to tracial von Neumann algebras (for self-adjoint $C$) thereby obtaining more general numerical ranges than those considered in \cites{AA2002, AA2003-1, AA2003-2, AAW1999}.   The $C$-numerical range of an operator is a compact, convex set defined using a notion of majorization for eigenvalue functions of self-adjoint operators and is described via an equation like equation (\ref{eq:matrix-C-nr}) inside II$_1$ factors.  Many properties of $C$-numerical ranges will be demonstrated including continuity results and  lack of dependence on the von Neumann algebra considered.

Throughout this paper, $(\fM,\tau)$ will denote a von Neumann algebra $\fM$ possessing a normal, faithful, tracial state,
with $\tau$ such a state.
We will call such a pair {\em a tracial von Neumann algebra}.
Furthermore, $\proj (\fM)$ will denote the set of projections in $\fM$ and $\fM_{\sa}$ will be used to denote the  set of
self-adjoint elements of $\fM$.

To begin, we will need a concept whose origin is due to Hardy, Littlewood, and P\'{o}lya.
\begin{defn}[see \cite{HLP1929}]
Let $f, g\in L^\infty[0,1]$.  It is said that $f$ \emph{majorizes} $g$, denoted $g \prec f$, if
\[
\int^t_0 g(x) \, dx \leq \int^t_0 f(x) \, dx \text{ for all }t \in [0,1] \qqand \int^1_0 g(x) \, dx = \int^1_0 f(x) \, dx.
\]
\end{defn}
Note if $g \prec f$ and $h \prec g$, one clearly has $h \prec f$.

To apply the above definition, we desire an analogue of eigenvalues for self-adjoint operators in tracial von Neumann algebras.
For this section and the rest of the paper, given an normal operator $N$ in a von Neumann algebra, we will use $1_{X}(N)$
to denote the spectral projection of $N$ corresponding to a Borel set $X \subseteq \bC$.  
\begin{defn}
\label{defn:spectral-scale}
Let $(\fM, \tau)$ be a diffuse, tracial von Neumann algebra and let $T \in \fM$ be self-adjoint.  The \emph{eigenvalue function of $T$} is defined for $s \in [0,1)$ by
\[
\lambda_T(s) = \inf \{t \in \bR \mid \tau(1_{(t,\infty)}(T)) \leq s \}.
\]
\end{defn}

It is elementary to verify that the eigenvalue function of $T$ is a bounded, non-increasing, right continuous function from $[0,1)$ to $\bR$.
The following result is seemingly folklore, and a proof may be found in \cite{AM2007}*{Proposition 2.3}.

\begin{prop} 
\label{prop:spectral-scale}
Let $(\fM, \tau)$ be a diffuse, tracial von Neumann algebra and let $T \in \fM$ be self-adjoint.
Then there is a projection-valued measure $e_T$ on $[0,1)$ valued in $\fM$ such that $\tau(e_T([0,t))) = t$ for every $t \in [0,1)$ and
\[
T = \int_0^1 \lambda_T(s)\, de_T(s).
\]
In particular $\tau(T) = \int^1_0 \lambda_T(s) \, ds$.
\end{prop}
\begin{rem}
\label{rem:L-infty}
Note the von Neumann algebra generated by $\{e_T([0,t))\}_{t \in [0,1)}$ is isomorphic to a copy of $L^\infty[0,1]$ inside $\fM$ in such a way that $T$ corresponds to the $L^\infty$-function $s \mapsto \lambda_T(s)$ and $\tau$ restricts to integration against the Lebesgue measure $m$.  
\end{rem}

 Using the above definitions, we may  now define the main objects of study in this paper.

\begin{defn}
\label{defn:numran}
Let $(\fM, \tau)$ be a tracial von Neumann algebra and let $C \in \fM_\sa$.  The \emph{$C$-numerical range} of an element $T \in \fM$ is the set
\[
V_C(T) := \{ \tau (TX) \, \mid \, X \in \fM_{\sa}, \lambda_X \prec \lambda_C\}.
\]
\end{defn}

\begin{rem}\label{rem:Valpha}
It is not difficult to verify that if $(\fM, \tau)$ is a tracial von Neumann algebra, if $T, S \in \fM_{\sa}$ with $T$ positive, and if $\lambda_S \prec \lambda_T$, then $S$ must be positive.  In addition, it is not difficult to show that if $P \in \fM$ is a projection with $\tau(P) = \alpha \in [0,1]$, then
\[
\{X \in \fM_{\sa} \, \mid \, \lambda_X \prec \lambda_P\} = \{X \in \fM \, \mid \, 0 \leq X \leq I_\fM, \tau(X) = \alpha\}.
\]
In analogy, for $\alpha \in (0,1]$ and $T \in \fM$, we define the \emph{$\alpha$-numerical range of} $T$ to be the set
\[
V_\alpha(T) := \frac{1}{\alpha} \{ \tau (TX) \, \mid \, X \in \fM, 0 \leq X \leq I_\fM, \tau(X) = \alpha\}.
\]
The $\alpha$-numerical ranges were originally studied (through a multivariate analogue for commuting $n$-tuples of self-adjoint operators) in the papers \cites{AA2002, AA2003-1, AA2003-2, AAW1999} and the $\frac{1}{\alpha}$  factor
is included so that if $0 < \alpha < \beta \leq 1$ then $V_\beta(T) \subseteq V_\alpha(T)$.  
\end{rem}

The following contains a collection of important properties of $C$-numerical ranges that mainly follow from properties of eigenvalue functions contained in \cites{F1982, FK1986, P1985}.  Note for two subsets $X, Y$ of $\bC$ and $\omega \in \bC$, we define
\begin{align*}
\omega X &= \{ \omega z \, \mid \, z \in X\}, \\
\omega + X &= \{ \omega + z \, \mid \, z \in X\}, \text{ and}\\
X + Y &= \{z + w \, \mid \, z \in X, w \in Y\}.
\end{align*}
\begin{prop}
\label{prop:basic-properties}
Let $(\fM, \tau)$ be a tracial von Neumann algebra, let $T, S \in \fM$, and let $C, C_1, C_2 \in \fM_{\sa}$.  Then
\begin{enumerate}[(i)]
\item $V_C(T)$ is a convex set for all $T \in \fM$,  \label{part:convex}
\item $V_C(T^*)$ equals the complex conjugate of $V_C(T)$, 
\item $V_C(\Re(T)) = \{\Re(z)\, \mid \, z \in V_C(T)\}$ and $V_C(\Im(T)) = \{\Im(z)\, \mid \, z \in V_C(T)\}$,\label{part:real}
\item $V_C(T + S) \subseteq V_C(T) + V_C(S)$, \label{part:sum}
\item $V_C( z I_\fM + w T) = z\tau(C) + w V_C(T)$ for all $z,w \in \bC$, \label{part:rotate-translate}
\item $V_C(U^*TU) = V_C(T)$ for all unitaries $U \in \fM$, \label{part:unitary}
\item $V_{C_1}(T) \subseteq V_{C_2}(T)$ whenever $C_1 \prec C_2$, and
\item $V_{aC+bI_\fM}(T) = aV_C(T) + b\tau(T)$ for all $a,b \in \bR$. \label{part:linear-in-C}
\end{enumerate}
\end{prop}
\begin{proof}
For part (\ref{part:convex}), notice that if $X_1, X_2 \in \fM_{\sa}$ are such that $\lambda_{X_1}, \lambda_{X_2} \prec \lambda_C$, then 
\[
\lambda_{t X_1 + (1-t) X_2} \prec t \lambda_{X_1} +(1-t) \lambda_{X_2} \prec \lambda_C
\]
for all $t \in [0,1]$ by \cite{FK1986}*{Lemma 2.5 (ii)}, by \cite{FK1986}*{Theorem 4.4}, and by a simple translation argument to assume all three operators are positive.  Hence it trivially follows that
\[
\{X \in \fM_{\sa} \, \mid \,  \lambda_X \prec \lambda_C\}
\]
is a convex set so $V_C(T)$ is convex (being the image under a linear map of a convex set).

Except for parts (\ref{part:unitary}) and (\ref{part:linear-in-C}), the other parts are trivial computations.  To see part (\ref{part:unitary}), note $\lambda_{U^*CU} = \lambda_C$ for all unitaries $U \in \fM$ and all $C \in \fM_{\sa}$.  To see part (\ref{part:linear-in-C}), note it is trivial to verify that $\lambda_{C + bI_\fM}(s) = \lambda_C(s) + b$ for all $s\in [0,1)$.  If $a \in \bR$ is positive, then $\lambda_{aC}(s) = a \lambda_C(s)$ for all $s \in [0,1)$.  Consequently, if $a$ is positive, then $\lambda_X \prec \lambda_C$ if and only if $\lambda_{aX} \prec \lambda_{aC}$ so the result follows.  If $a \in \bR$ is negative, then one can verify that $\lambda_{aC}(s) = a \lambda_C(1-s)$ for all but a countable number of $s \in [0,1)$ where the jump discontinuities of $\lambda_C(s)$ occur.  One can again verify in this case that $\lambda_X \prec \lambda_C$ if and only if $\lambda_{aX} \prec \lambda_{aC}$ so the result follows.
\end{proof}

Our next goal is to show the very useful property that the $C$-numerical ranges of an operator do not depend on the ambient von Neumann algebra.  To do so, we recall the following result.

\begin{thm}[see \cite{AK2006}]
\label{thm:AK-expect}
Let $(\fM, \tau)$ be a tracial von Neumann algebra, let $\fN$ be a von Neumann subalgebra of $\fM$, and let $E_\fN : \fM \to \fN$ be the trace-preserving conditional expectation of $\fM$ onto $\fN$.  Then $\lambda_{E_\fN(X)} \prec \lambda_{X}$ for all $X \in \fM_{\sa}$.
\end{thm}

\begin{prop}
\label{prop:algebra-doesnt-affect-nr}
Let $(\fM,\tau)$ be a tracial von Neumann algebra and let $C \in \fM_{\sa}$.
For $T\in\fM$ let $V_C(T)$ denote the $C$-numerical range as given in Definition \ref{defn:numran}.
Let $\fN$ be a von Neumann subalgebra of $\fM$ such that $T \in \fN$. Then
\begin{align}
V_C(T) = \{\tau(TX) \, \mid \, X \in \fN_{\sa}, \lambda_X \prec \lambda_C\}. \label{eq:inN}
\end{align}
In particular, $V_C(T)$ does not depend on the diffuse tracial von Neumann algebra considered.
\end{prop}
\begin{proof}
The inclusion $\supseteq$ in (\ref{eq:inN}) is clear.
For the reverse inclusion,
let $E_\fN : \fM \to \fN$ denote the trace-preserving conditional expectation of $\fM$ onto $\fN$.
If $X \in \fM_{\sa}$ is such that $\lambda_X \prec \lambda_C$, then $E_{\fN}(X) \in \fN$, $\lambda_{E_\fN(X)} \prec \lambda_{X} \prec \lambda_C$ by Theorem \ref{thm:AK-expect}, and
\[
\tau(TE_\fN(X)) = \tau(E_\fN(TX)) = \tau(TX).
\]
This proves (\ref{eq:inN}).
\end{proof}

By Proposition \ref{prop:algebra-doesnt-affect-nr}, we may compute the $C$-numerical ranges in any tracial von Neumann algebra we like.  In particular, as every tracial von Neumann algebra embeds in a trace-preserving way into a type II$_1$ factor, we may restrict our attention to type II$_1$ factors when considering $C$-numerical ranges.  By doing so, we will obtain an alternate description of $C$-numerical ranges that is a direct analogue of equation (\ref{eq:matrix-C-nr}) and produces many corollaries.  We begin with the following.

\begin{defn}
Let $\fA$ be an arbitrary C$^*$-algebra and let $\U(\fA)$ denote the unitary group of $\fA$.
For $T \in \fA$, the \emph{unitary orbit of $T$} is the set
\[
\U(T) = \{U^*TU \, \mid \, U \in \U(\fA)\}
\]
and the (norm-)closed unitary orbit of $T$ is the set $\O(T) = \overline{\U(T)}^{\left\|\,\,\cdot\,\,\right\|}$.
\end{defn}

\begin{rem}
\label{rem:self-adjoints-aue}
Notice if $T, S \in \fM$ are self-adjoint operators then $\lambda_T \prec \lambda_S$ and $\lambda_S \prec \lambda_T$ if and only of $\lambda_T(s) = \lambda_S(s)$ for all $s \in [0,1)$.  By Definition \ref{defn:spectral-scale}, these are equivalent to $T$ and $S$ having the same spectral distribution.  It is well-known that these are all equivalent to $T \in \O(S)$, provided $\fM$ is a type II$_1$ factor.
\end{rem}

Notice that if $\fA$ is a finite dimensional  C$^*$-algebra, then $\U(T) = \O(T)$.  In general, $\O(T)$ is the correct object to consider when studying infinite dimensional C$^*$-algebras.  In particular, we will use $\O(T)$ to generalize equation (\ref{eq:matrix-C-nr}) to type II$_1$ factors.  In particular, the work of \cites{GS1977, P1980} proves the following result when $\fM$ is a matrix algebra.
\begin{thm}
\label{thm:C-numerical-ranges-alternate-defn}
Let $(\fM, \tau)$ be a type II$_1$ factor and let $C \in \fM_{\sa}$.   Then for all $T \in \fM$,
\[
V_C(T) = \{ \tau(TX) \, \mid \, X \in \fM_{\sa}, X \in \O(C)\}.
\]
\end{thm}

To prove Theorem \ref{thm:C-numerical-ranges-alternate-defn}, we will need two results.  The first is the following connection between majorization of eigenvalue functions and convex hulls of unitary orbits.

\begin{thm}[see \cites{AM2008, A1989, H1987, HN91, K1983, K1984, K1985}]
\label{thm:majorization}
Let $(\fM, \tau)$ be a factor and let $X, T \in \fM_{\sa}$.  Then the following are equivalent:
\begin{enumerate}
\item $\lambda_X \prec \lambda_T$.
\item $X \in \overline{\mathrm{conv}(\U(T))}^{\left\|\,\cdot\,\right\|}$.  
\item $X \in \overline{\conv(\U(T))}^{w^*}$.  
\end{enumerate}
\end{thm}

The second result required to prove Theorem \ref{thm:C-numerical-ranges-alternate-defn} is the following technical result, whose proof is contained in the proof of \cite{DFHS2012}*{Theorem 5.3} and follows by simple manipulations of functions.

\begin{prop}[\cite{DFHS2012}*{Theorem 5.3}]
\label{prop:strictly-majorized-important-technical-lemma}
Let $(\fM, \tau)$ be a type II$_1$ factor and let $A, C \in \fM$ be self-adjoint operators such that $\lambda_A \prec \lambda_C$ and $A \notin \O(C)$.
Then there exists a non-zero projection $P \in \fM$ and an $\epsilon > 0$ such that $\lambda_{A + S} \prec \lambda_C$
for all self-adjoint operators $S \in \fM$
satisfying $\left\|S\right\| < \epsilon$, $S = PS = SP$, and $\tau(S) = 0$.
\end{prop}

\begin{proof}[\textbf{Proof of Theorem \ref{thm:C-numerical-ranges-alternate-defn}}]
Fix $C \in \fM_{\sa}$ and $T \in \fM$.
Then
\[
\{ \tau(TX) \, \mid \, X \in \fM_{\sa}, X \in \O(C)\} \subseteq V_C(T)
\]
by Remark \ref{rem:self-adjoints-aue} and Definition \ref{defn:numran}.  

For the other inclusion, fix $X \in \fM_{\sa}$ with $\lambda_X \prec \lambda_C$ and define
\[
Q_{X, C} = \{Y \in \fM_{\sa} \, \mid \, \tau(TY) = \tau(TX),\, \lambda_Y \prec \lambda_C\}.
\]
Since the linear map $Z \mapsto \tau(TZ)$ is weak$^*$-continuous,
by Theorem \ref{thm:majorization}
$Q_{X, C}$ is a non-empty (as $X \in Q_{X, C}$), convex, weak$^*$-compact subset.
Hence, by the Krein--Milman Theorem, $Q_{X, C}$ has an extreme point, say $A$.

We will show $A \in \O(C)$ to complete the proof.
To see this, suppose to the contrary that $A \notin \O(C)$.  Since $A \in Q_{X, C}$, $\lambda_A \prec \lambda_C$ so by Proposition \ref{prop:strictly-majorized-important-technical-lemma} there exists a non-zero projection $P \in \fM$ and an $\epsilon > 0$ such that $\lambda_{A + S} \prec C$ for all self-adjoint operators $S \in \fM$ with $\left\|S\right\| < \epsilon$, $S = PS = SP$, and $\tau(S) = 0$.

Consider the linear map 
\[
\psi : \{S \in \fM_{\sa} \, \mid \, S = PS = SP, \tau(S) = 0 \} \to \bC
\]
defined by $\psi(S) = \tau(TS)$.  By dimension requirements, there exists a $S \in \ker(\psi)\setminus \{0\}$.  By scaling, we obtain a non-zero $S \in \fM_{\sa}$ such that $\left\|S\right\| < \epsilon$, $S = PS = SP$, $\tau(S) = 0$, and $\tau(TS) = 0$.  By construction $A \pm S \in Q_{X, C}$ and since
\[
A  =\frac{1}{2}(A + S) + \frac{1}{2}(A-S)
\]
we obtain a contradiction to the fact that $A$ was an extreme point of $Q_{X,C}$. 
\end{proof}

With Proposition \ref{prop:algebra-doesnt-affect-nr} and Theorem \ref{thm:C-numerical-ranges-alternate-defn} complete, we obtain several important corollaries.   In fact, \cite{AA2002} went to great lengths to obtain a (multivariate) version of the following result, for which our techniques provide a quicker proof.
\begin{cor}
\label{cor:projection-nr}
Let $(\fM, \tau)$ be a type II$_1$ factor, let $T \in \fM$, and let $\alpha\in (0, 1]$.  Then
\[
V_\alpha(T) = \frac{1}{\alpha} \{ \tau(TP) \, \mid \, P \in \proj(\fM), \tau(P) = \alpha\}.
\]
\end{cor}

\begin{cor}
\label{cor:compact}
Let $(\fM, \tau)$ be a tracial von Neumann algebra,  let $T \in \fM$, and let $C \in \fM_{sa}$.  Then $V_C(T)$ is a compact set.
\end{cor}
\begin{proof}
By Proposition \ref{prop:algebra-doesnt-affect-nr} we may assume that $\fM$ is a type II$_1$ factor.  Hence Theorem \ref{thm:AK-expect} implies that
\[
V_C(T) = \left\{ \tau(TX) \, \left| \, X \in \overline{\conv(\U(T))}^{w^*} \right. \right\}.
\]
As $\overline{\conv(\U(T))}^{w^*}$ is weak$^*$-compact and $\tau$ is a weak$^*$-continuous linear functional, we obtain that $V_C(T)$ is compact.
\end{proof}

\begin{cor}
Let $(\fM, \tau)$ be a tracial von Neumann algebra and let $T, C \in \fM_{sa}$.  Then $V_C(T) = V_T(C)$.
\end{cor}
\begin{proof}
By Proposition \ref{prop:algebra-doesnt-affect-nr} we may assume that $\fM$ is a type II$_1$ factor.  As $\U(T)$ is (norm-)dense in $\O(T)$ and $\U(C)$ is (norm-)dense in $\O(C)$, we obtain that
\[
\{\tau(TU^*CU) \, \mid \, U \in \fM, U \text{ a unitary}\}
\]
is dense in both $V_C(T)$ and $V_T(C)$ by Theorem \ref{thm:C-numerical-ranges-alternate-defn}.  Hence $V_C(T) = V_T(C)$ as both sets are compact by Corollary \ref{cor:compact}.
\end{proof}

Another important corollary is the continuity of the $C$-numerical range of $T$ as both $C$ and $T$ vary. For this discussion, recall that for compact subsets $X$ and $Y$ of $\bC$, the Hausdorff distance between $X$ and $Y$ is
defined to be
\[
d_H(X, Y) = \max\left\{  \sup_{x \in X} \mathrm{dist}(x, Y), \sup_{y \in Y} \mathrm{dist}(y, X) \right\}.
\]

\begin{prop}
\label{prop:continuity-in-alpha}
Let $(\fM, \tau)$ be a tracial von Neumann algebra and let $T\in \fM$.  If $C_1, C_2 \in \fM_{\sa}$, then
\begin{align*}
d_H(V_{C_1}(T), V_{C_2}(T)) \leq \left\|T\right\|\left\|C_1 - C_2\right\|.
\end{align*}
In particular, the map $C \mapsto V_C(T)$ is a continuous map from $\fM_{\sa}$ (equipped with the operator norm) to the compact, convex subsets of $\bC$ equipped with the Hausdorff distance.
\end{prop}
\begin{proof}
To begin we may assume that $\fM$ is a type II$_1$ factor by Proposition \ref{prop:algebra-doesnt-affect-nr}.  Note for all $X \in \O(C_1)$ and $\epsilon > 0$ there exists an $X' \in \O(C_2)$ such that 
\[
\left\|X - X'\right\| \leq \epsilon + \left\|C_1 - C_2\right\|
\]
and thus
\[
\left| \tau(TX) - \tau(TX')\right| \leq \left\|T\right\|\left\|X - X'\right\| \leq \left\|T\right\|\left\|C_1 - C_2\right\| + \epsilon \left\|T\right\|.
\]
As one may also interchange the roles of $C_1$ and $C_2$, the result follows by Theorem \ref{thm:C-numerical-ranges-alternate-defn}.
\end{proof}

\begin{prop}
\label{prop:continuity-in-nr}
Let $(\fM, \tau)$ be a tracial von Neumann algebra, let $T, S\in \fM$, and let $C \in \fM_{\sa}$.  Then
\[
d_H(V_C(T), V_C(S)) \leq \left\|C\right\|\left\|T-S\right\|.
\]
Thus, for any fixed $C \in \fM_{sa}$, the map $T \mapsto V_C(T)$ is continuous from $\fM$ (equipped with the operator norm) to the compact, convex subsets of $\bC$ equipped with the Hausdorff distance.
\end{prop}
\begin{proof}
To begin we may assume that $\fM$ is a type II$_1$ factor by Proposition \ref{prop:algebra-doesnt-affect-nr}.  For all $X \in \O(C)$, notice
\[
\left|\tau(TX) - \tau(SX)\right| \leq \left\|T-S\right\|\left\|X\right\| = \left\|T-S\right\|\left\|C\right\|.
\]
Hence the result follows by Theorem \ref{thm:C-numerical-ranges-alternate-defn}.
\end{proof}

\begin{cor}
\label{cor:aue-implies-nr-equal}
Let $(\fM, \tau)$ be a tracial von Neumann algebra and let $T, S \in \fM$.  If $T$ and $S$ are approximately unitarily equivalent, that is $S \in \O(T)$, then $V_C(T) = V_C(S)$ for all $C \in \fM_{\sa}$.
\end{cor}
\begin{proof}
The result follows from part (\ref{part:unitary}) of Proposition \ref{prop:basic-properties} and Proposition \ref{prop:continuity-in-nr}.
\end{proof}

\section{$C$-Numerical Ranges of Self-Adjoint Operators}
\label{sec:SA}

In this section, we will use eigenvalue functions to describe $V_C(T)$ when $C, T \in \fM_{\sa}$.  This will be of use in the subsequent section when developing a method for computing $C$-numerical ranges of an arbitrary operator $T$.  

To begin our description of $V_C(T)$ for all $C, T \in \fM_{\sa}$, we will assume that $C$ and $T$ are positive operators.
From the description of such $V_C(T)$, Proposition \ref{prop:basic-properties} will yield descriptions of $V_C(T)$ for all $C, T \in \fM_{\sa}$.
\begin{prop}
\label{prop:self-adjoints}
Let $(\fM, \tau)$ be a tracial von Neumann algebra and let $T, C \in \fM$ be positive.   Then
\[
V_C(T) = \left[ \int^1_0 \lambda_T(s) \lambda_C(1-s) \, ds, \int^1_{0} \lambda_T(s) \lambda_C(s)\, ds  \right].
\]

\end{prop}
\begin{rem}
\label{rem:num-for-sa}
Note if $T, C \in \fM_{\sa}$ with $C$ positive, then we still have
\[
V_C(T) = \left[ \int^1_0 \lambda_T(s) \lambda_C(1-s) \, ds, \int^1_{0} \lambda_T(s) \lambda_C(s)\, ds  \right]
\]
by Proposition \ref{prop:basic-properties} and the fact that $\lambda_{a I_\fM + T}(s) = a + \lambda_T(s)$ for all $s \in [0,1)$ and $a \in \bR$.  
\end{rem}

To begin the proof of Proposition \ref{prop:self-adjoints}, we note by Remark \ref{rem:L-infty} and Proposition \ref{prop:algebra-doesnt-affect-nr} that we may assume $\fM = L^\infty[0,1]$ equipped with
the trace given by integration against
Lebesgue measure $m$ and that $T = \lambda_T$ as a function on $[0,1]$.  

To understand $C$-numerical ranges inside $L^\infty[0,1]$, we need to understand which functions have the same eigenvalue functions.  This returns us to the work of Hardy, Littlewood, and P\'{o}lya.

\begin{defn}[\cite{HLP1929}*{Section 10.12}]
\label{defn:rearrangement}
For a real-valued function $f \in L^\infty[0,1]$, the \emph{non-increasing rearrangement} of $f$ is the function
\[
f^*(s) = \inf \{ x \mid m(\{ t \mid f(t) \geq x \}) \leq s\} \text{ for all } s \in [0,1].
\]
\end{defn}

It is not difficult to show that if $f \in L^\infty[0,1]$, then $\lambda_f = f^*$.  Consequently $f^*$ is non-increasing, right-continuous function on $[0,1)$ that is positive when $f$ is positive.  Furthermore, if $f$ is a characteristic function, it is not difficult to see how $f^*$ is a rearrangement of $f$ into a non-increasing function.

We begin the demonstration of Proposition \ref{prop:self-adjoints} with the following.
\begin{lem}
\label{lem:non-increasing}
Let $f,g \in L^\infty[0,1]$ be non-increasing, positive, right continuous functions where $g$ is a step function.  Then
\[
\int^1_0 f(x) g(x) \, dx = \sup\left\{ \left.\int^1_0 f(x) h(x)\, dx \, \right| \, h^* = g  \right\}.
\]
and
\[
\int^1_0 f(x) g(1-x) \, dx = \inf\left\{ \left. \int^1_0 f(x) h(x)\, dx \, \right| \, h^* = g  \right\}.
\]
\end{lem}
\begin{proof}
By the assumptions on $g$, there exists $0 = x_0 < x_1 < \cdots < x_n = 1$ and $a_1 > a_2 > \cdots > a_n \geq 0$ such that
\[
g = \sum^n_{k=1} a_k 1_{[x_{k-1}, x_k)}.
\]

Suppose $h \in L^\infty[0,1]$ is such that $h^* = g$.  By the definition of the non-increasing rearrangement (also see Remark \ref{rem:self-adjoints-aue}), there exists disjoint Borel subsets $\{X_k\}^n_{k=1}$ of $[0,1]$ such that $m\left( \bigcup^n_{k=1} X_k\right) = 1$, $m(X_k) = x_k - x_{k-1}$ for all $k$, and 
\[
h = \sum^n_{k=1} a_k 1_{X_k}.
\]
We claim that
\[
\int^1_0 f(x) h(x) \, dx \leq \int^1_0 f(x) g(x) \, dx.
\]
To see this, suppose $h \neq g$.  Let $k(h)$ be the smallest index so that 
\[
m([x_{k(h)-1}, x_{k(h)}) \setminus X_{k(h)}) > 0.
\]
By the selection of $k(h)$ and since $m\left( \bigcup^n_{k=1} X_k\right) = 1$ and $m(X_k) = x_k - x_{k-1}$, there exists a smallest $k'(h) > k(h)$ so that 
 \[
 m([x_{k(h)-1}, x_{k(h)}) \cap X_{k'(h)}) > 0.
 \]
Furthermore, $m(X_{k(h)} \setminus [x_{k(h)-1}, x_{k(h)}))  = m([x_{k(h)-1}, x_{k(h)}) \setminus X_{k(h)}) > 0$ as $m(X_k) = x_k - x_{k-1}$, and
\[
X_{k(h)} \setminus [x_{k(h)-1}, x_{k(h)}) \subseteq [x_{k(h)}, 1]
\]
by the selection of $k(h)$.  Therefore, there exists $Y \subseteq X_{k(h)} \setminus [x_{k(h)-1}, x_{k(h)})$ and $Z \subseteq [x_{k(h)-1}, x_{k(h)}) \cap X_{k'(h)}$ so that
\[
m(Y) = m(Z) = \min\{m([x_{k(h)-1}, x_{k(h)}) \setminus X_{k(h)}), m([x_{k(h)-1}, x_{k(h)}) \cap X_{k'(h)})    \}.
\]
If $X^1_{k(h)} := Z \cup (X_{k(h)} \setminus Y)$, $X^1_{k'(h)} := Y \cup (X_{k'(h)} \setminus Z)$, $X^1_k := X_k$ when $k \neq k(h), k'(h)$, and
\[
h_1 = \sum^n_{k=1} a_k 1_{X^1_k},
\]
then it is elementary to verify that $(h_1)^* = h^* = g$.   Furthermore
\begin{align*}
\int^1_0 f(x)(h_1(x) - h(x)) \, dx &= \int_{Z} f(x) (a_{k(h)} - a_{k'(h)}) \, dx + \int_Y f(x) (a_{k'(h)} - a_{k(h)}) \, dx\\
&= (a_{k(h)} - a_{k'(h)}) \left( \int_{Z} f(x)  \, dx - \int_Y f(x) \, dx   \right) \geq 0
\end{align*}
since $a_{k(h)} - a_{k'(h)} \geq 0$, $m(Z) = m(Y)$, $\sup(Z) \leq \inf(Y)$, and $f$ is a positive, non-increasing function.  

If $h_1 \neq g$, then one can repeat the above arguments with $h_1$ in place of $h$ where one necessarily has either $k(h_1) > k(h)$ or $k(h_1) = k(h)$ and $k'(h_1) > k'(h)$.  As there are a finite number of indices, one eventually constructs $h = h_0, h_1, h_2, \ldots, h_{m-1}, h_m = g$ with $(h_j)^* = g$ and
\[
\int^1_0 f(x) h_j(x) \, dx \leq \int^1_0 f(x) h_{j+1}(x) \, dx
\]
for all $j$.  Hence, as $h \in L^\infty[0,1]$ with $h^* = g$ was arbitrary, we obtain 
\[
\int^1_0 f(x) g(x) \, dx = \sup\left\{ \left.\int^1_0 f(x) h(x)\, dx \, \right| \, h^* = g  \right\}.
\]

The other equation in the statement of the result is proved using similar techniques.
\end{proof}

\begin{proof}[\textbf{Proof of Proposition \ref{prop:self-adjoints}}]
As remarked above, we may assume $\fM = L^\infty[0,1]$ and $T = \lambda_T$ under this identification.  Since the map $X \mapsto \lambda_X$ is operator-norm to $L^\infty[0,1]$-norm continuous, and since $T \mapsto V_C(T)$ and $C \mapsto V_C(T)$ are operator-norm to Hausdorff distance continuous, we may assume without loss of generality that $T$ and $C$ have finite spectrum.  Consequently, there exists $0 = x_0 < x_1 < \cdots < x_n = 1$, $t_1 \geq t_2 > \cdots > t_n \geq 0$, and $c_1 \geq c_2 \geq \cdots \geq c_n \geq 0$ such that
\[
T = \sum^n_{k=1} a_k 1_{[x_{k-1}, x_k)} \qqand \lambda_C = \sum^n_{k=1} c_k 1_{[x_{k-1}, x_k)}.
\]
As $\lambda_C \in \fM$ and
\[
\tau(T \lambda_C) = \int^1_0 \lambda_T(x) \lambda_C(x) \, dx
\]
by definition, we clearly have $\int^1_0 \lambda_T(x) \lambda_C(x) \, dx \in V_C(T)$.  Similarly,
letting  $f(x) = \lambda_C(1-x)$,  we have $f\in \fM$, $f^* = \lambda_C$, and
\[
\tau(Tf) = \int^1_0 \lambda_T(x) \lambda_C(1-x) \, dx,
\]
we clearly have $\int^1_0 \lambda_T(x) \lambda_C(1-x) \, dx \in V_C(T)$.  Since $V_C(T)$ is a compact, convex subset of $\bR$ (as $C$ and $T$ are positive), it suffices so show that
\[
\sup(V_C(T)) = \int^1_0 \lambda_T(x) \lambda_C(x) \, dx \qand \inf(V_C(T)) = \int^1_0 \lambda_T(x) \lambda_C(1-x) \, dx
\]
to complete the proof.

Suppose $g \in \fM$ is such that $\lambda_g \prec \lambda_C$ (thus $g$ is positive).  We desire to show that $\tau(Tg) \leq \tau(T \lambda_C)$.  Let $\fN$ be the von Neumann subalgebra of $\fM$ generated by the projections $\{1_{[x_{k-1}, x_k)}\}_{k=1}^n$ and let $E_{\fN} : \fM \to \fN$ be the trace-preserving conditional expectation onto $\fN$.  By Theorem \ref{thm:AK-expect}, $h = E_\fN(g) \in \fN$ is a positive operator with finite spectrum such that $\lambda_h \prec \lambda_g \prec \lambda_C$ and $\tau(Th) = \tau(Tg)$.  Hence it suffices to show $\tau(Tg) \leq \tau(T \lambda_C)$ for all $g \in \fM$ with finite spectrum and $\lambda_g \prec \lambda_C$.

For such a $g$, we may without loss of generality assume $g = g^*$ by Lemma \ref{lem:non-increasing}.  Consequently, we may assume there exists $0 = x'_0 < x'_1 < \cdots < x'_m = 1$, $a_1 \geq a_2 \geq \cdots \geq a_m \geq 0$,
$c'_1 \geq c'_2 \geq \cdots \geq c'_n \geq 0$, and $b_1 \geq b_2 \geq \cdots \geq b_m \geq 0$ such that
\[
T = \sum^m_{k=1} a'_k 1_{[x'_{k-1}, x'_k)}, \quad \lambda_C = \sum^m_{k=1} c'_k 1_{[x'_{k-1}, x'_k)}, \qand g = \sum^m_{k=1} b_k 1_{[x'_{k-1}, x'_k)}.
\]
Since $g \prec \lambda_C$, we obtain that
\begin{equation}\label{eq:bxcx}
\sum^q_{k=1} b_k (x'_{k} - x'_{k-1}) \leq \sum^q_{k=1} c'_k (x'_{k} - x'_{k-1})
\end{equation}
for all $q$ with equality when $q = m$.  Therefore,  setting $a'_{m+1} = 0$,  we have
\begin{align*}
\tau(T(\lambda_C - g)) &= \sum^m_{k=1} a'_k(c'_k - b_k)(x'_{k} - x'_{k-1})\\
&= \sum^m_{q=1}  \sum^q_{j=1}  \left(a'_q - a'_{q+1}   \right)   (c'_j - b_j)(x'_{j} - x'_{j-1}).
\end{align*}
Since $a'_q - q'_{q+1} \geq 0$ for all $q$ and $\sum^q_{j=1}  (c'_j - b_j)(x'_{j} - x'_{j-1}) \geq 0$ by  \eqref{eq:bxcx}, we obtain $\tau(T(\lambda_C - g)) \geq 0$ as desired.

The proof that 
\[
\inf(V_C(T)) = \int^1_0 \lambda_T(x) \lambda_C(1-x) \, dx
\]
follows from similar arguments.
\end{proof}

\section{A Method for Computing $C$-Numerical Ranges}
\label{sec:Computing}

In this section, we will use Proposition \ref{prop:self-adjoints} together with some additional arguments to develop a method for computing $V_C(T)$ for general $T \in \fM$.  This will enable us to show that if one knows all $\alpha$-numerical ranges of an operator $T$, one also knows all $C$-numerical ranges of $T$.   

Given an operator $T$, the main idea is to reduce the computation of the $C$-numerical range of $T$ to the $C$-numerical ranges of the real parts of rotations of $T$, which are described in terms of eigenvalue functions by Proposition \ref{prop:self-adjoints}.  This is motivated by \cite{K1951} (or see the English translation \cite{K2008}).  To begin, we will require the following functions.

\begin{nota}
For a non-empty, bounded subset $E\subseteq\bC$, let
\[
\sup (\Re (E)) =\sup\{\Re (z) \, \mid \, z\in E\} 
\]
and define $g_E:[0,2\pi)\to\bR$ by
\[
g_E(\theta)=\sup (\Re(e^{i\theta}E)).
\]
\end{nota}

\begin{prop}
\label{prop:gK}
For a non-empty, compact, convex set $K \subseteq \bC$, the function $g_K$ completely determines $K$.  In particular
\[
K=\{z\in\bC \, \mid \, \Re(e^{i\theta}z)\leq g_K(\theta) \text{ for all }\theta\in[0,2\pi)\}.
\]
\end{prop}

\begin{proof}
Let $\Psi(K)$ denote the set on the right-hand-side of the above equation.  Since $g_{w+K}(\theta)=\Re(e^{i\theta}w)+g_K(\theta)$ for all $w \in \bC$, we have
\[
\Psi(w + K) = w + \Psi(K).
\]
Thus, we may assume without loss of generality that $0 \in K$.

By definition, it is clear that $K \subseteq \Psi(K)$.  For the other inclusion, suppose $w\in K^c$.  By the Hahn-Banach Theorem there is a line separating $w$ from $K$.
This line is the solution set in $\bC$ of the equation $\Re(e^{-i\theta}z)=c$ for some $\theta\in[0,2\pi)$ and some $c\ge0$.
Thus, the line $\Re(z)=c$ separates $e^{i\theta}K$ from $e^{i\theta}w$.
Since $0\in K$, we have that $0\leq g_K(\theta)<c<\Re(e^{i\theta}w)$ so $w \notin \Psi(K)$.
\end{proof}

\begin{exam}
\label{exam:ellipse-function}
For $a,b \in \bR$ with $a,b > 0$, consider the solid ellipse
\[
K=\left\{x+iy \, \left| \, x,y \in \bR, \frac{x^2}{a^2}+\frac{y^2}{b^2}\leq 1 \right.\right\}.
\]
The parametrization of the boundary of $K$ in polar coordinates is defined by the map
\[
\theta \mapsto a \cos(\theta) + i b \sin(\theta),
\]
and from this it is elementary to verify that
\[
g_K(\theta)=\sqrt{a^2\cos^2\theta+b^2\sin^2\theta}.
\]
\end{exam}

As the $C$-numerical ranges of an operator are compact, convex subsets of $\bC$,
in order to determine them
it suffices to describe the functions $g_{V_C(T)}(\theta)$.  Furthermore, it suffices to describe $V_C(T)$ for $C$ positive by part (\ref{part:linear-in-C}) of Proposition \ref{prop:basic-properties} (otherwise we translate $C$ to be a positive operator $C'$, compute $V_{C'}(T)$, and then translate back).

\begin{method}
\label{meth:g}
Given a tracial von Neumann algebra $(\fM, \tau)$, $T \in \fM$, and a positive $C \in \fM$, by combining Propositions \ref{prop:self-adjoints} and \ref{prop:gK} we obtain a method of computing $V_C(T)$, provided we can obtain sufficient information about the distributions of the operators $\Re(e^{i\theta}T)$ for $\theta\in[0,2\pi)$.
Indeed, by Proposition \ref{prop:self-adjoints} (or, more specifically, Remark \ref{rem:num-for-sa}), we have
\[
g_{V_C(T)}(\theta) = \int^1_0 \lambda_{\Re(e^{i\theta}T)}(s) \lambda_C(s) \, ds.
\]
Thus, Proposition \ref{prop:gK} implies
\[
V_\alpha(T) = \left\{z\in\bC \, \left| \, \Re(e^{i\theta}z)\leq \int^1_0 \lambda_{\Re(e^{i\theta}T)}(s) \lambda_C(s) \, ds \text{ for all }\theta\in[0,2\pi) \right. \right\}.
\]
\end{method}

In particular, the above method works provided we can describe $\lambda_C$ and $\lambda_{\Re(e^{i\theta}T)}$ for all $\theta \in [0, 2\pi)$.  In fact, the following theorem demonstrates it suffices to know $\lambda_C$ for all $C$ in $\proj(\fM)$.

\begin{thm}
\label{thm:alpha-ranges-enough}
Let $(\fM, \tau)$ be a tracial von Neumann algebra and let $T \in \fM$.  Then $\{(C, V_C(T)) \, \mid \, C \in \fM_{sa}\}$ is determined by $\{(P, V_P(T)) \, \mid \, P \in \proj(\fM)\}$.  In particular, the $C$-numerical ranges of an operator are determined by the $\alpha$-numerical ranges of an operator.
\end{thm}
\begin{proof}
By Method \ref{meth:g} it suffices to prove the result for $T \in \fM_{\sa}$.  Furthermore, by part (\ref{part:linear-in-C}) of Proposition \ref{prop:basic-properties} and by Proposition \ref{prop:continuity-in-alpha}, it suffices to show that if $C \in \fM$ is positive with a finite spectrum, then $V_C(T)$ is determined by $\{(P, V_P(T)) \, \mid \, P \in \proj(\fM)\}$.  

As $C \in \fM$ has finite spectrum, there exists pairwise orthogonal projections $\{P_k\}^n_{k=1} \subseteq \fM$ and $c_1 > c_2 > \cdots > c_n \geq 0$ such that
\[
C = \sum^n_{k=1} c_k P_k.
\]
It is elementary to show that if $x_0 = 0$ and $x_k = x_{k-1} + \tau(P_k)$ for all $k \geq 1$, then
\[
\lambda_C = \sum^n_{k=1} c_k 1_{[x_{k-1}, x_k)}.
\]
Consequently, by Remark \ref{rem:num-for-sa},
\[
V_C(T) = \left[ \sum^n_{k=1} c_{n-k+1} \int^{x_k}_{x_{k-1}} \lambda_T(x)  \, dx,  \sum^n_{k=1} c_{k} \int^{x_k}_{x_{k-1}} \lambda_T(x)  \, dx   \right].
\]
Consequently, if one knows $\int^{x_k}_{x_{k-1}} \lambda_T(x)  \, dx$ for all $k$, then one knows $V_C(T)$.

We claim that  each $\int^{x_k}_{x_{k-1}} \lambda_T(x)  \, dx$ is determined by $\{(P, V_P(T)) \, \mid \, P \in \proj(\fM)\}$.  Indeed if $Q_m= \sum^m_{k=1} P_k$, then $Q_m$ is a projection with $\tau(Q_m) = \sum^m_{k=1} \tau(P_k) = x_m$ and
\[
\int^{x_m}_0 \lambda_T(x)\, dx = \sup(V_{Q_m}(T))
\]
by Remark \ref{rem:num-for-sa}.  Hence
\[
\int^{x_k}_{x_{k-1}} \lambda_T(x)  \, dx = \sup(V_{Q_k}(T)) - \sup(V_{Q_{k-1}}(T))
\]
for all $k$ thereby completing the  proof of the claim.
\end{proof}

\section{Further Examples}
\label{sec:Examples}

Theorem \ref{thm:alpha-ranges-enough} demonstrates the $\alpha$-numerical ranges determine all $C$-numerical ranges.
In this section, we compute the $\alpha$-numerical ranges of several operators.
Although computing the $k$-numerical ranges of a matrix is generally a hard task, there are several interesting examples of operators in II$_1$ factor whose $\alpha$-numerical ranges can be explicitly described.

We begin by noting the following.
\begin{prop}
\label{prop:distributions}
Let $(\fM_1, \tau_1)$ and $(\fM_2, \tau_2)$ be tracial von Neumann algebras, let $T_1 \in \fM_1$, and let $T_2 \in \fM_2$.  If $T_1$ and $T_2$ have the same $*$-distributions, then $V_\alpha(T_1) = V_\alpha(T_2)$ for all $\alpha \in (0, 1]$.
\end{prop}
\begin{proof}
By Proposition \ref{prop:algebra-doesnt-affect-nr}, we may assume, without loss of generality, that $\fM_k = W^*(T_k)$ for $k = 1,2$.  Since $T_1$ and $T_2$ have the same $*$-distributions, there exists a trace-preserving isomorphism of $W^*(T_1)$ and $W^*(T_2)$ that sends $T_1$ to $T_2$.    This clearly implies $V_\alpha(T_1) = V_\alpha(T_2)$ for all $\alpha \in (0, 1]$, by Definition \ref{defn:numran}.
\end{proof}

Recall from the introduction that the $k$-numerical range of a normal matrix $N\in \M_n(\bC)$ with eigenvalues $\{\lambda_j\}^n_{j=1}$ is 
\[
W_k(N) = \mathrm{conv}\left(\left\{ \left. \frac{1}{k}\sum_{j \in K} \lambda_j \, \right| \, J \subseteq \{1,\ldots, n\}, |J| = k \right\}\right).
\]
 The following generalizes this result to normal operators with finite spectrum in a tracial von Neumann algebra.
\begin{prop}
\label{prop:normals-nr}
Let $(\fM, \tau)$ be a tracial von Neumann algebra, let $N \in \mathfrak{M}$ be a normal operator such that $\sigma(N) = \{\lambda_k\}^n_{k=1}$, and let $w_k = \tau(1_{\{\lambda_k\}}(N))$ for each $k \in \{1,\ldots, n\}$.
Then for each $\alpha \in (0,1]$, we have
\[
V_\alpha(N) = \left\{ \left. \frac{1}{\alpha} \sum^n_{k=1} \lambda_k t_k \, \right| \, 0 \leq t_k \leq w_k, \sum^n_{k=1} t_k = \alpha\right\}.
\]
\end{prop}
\begin{proof}
Using Proposition~\ref{prop:distributions}, we may without loss of generality assume $\fM = L^\infty[0,1]$ and
\[
N = \sum^n_{k=1} \lambda_k 1_{X_k}
\]
where $\{X_k\}^n_{k=1}$ are disjoint Borel measurable sets such that $\bigcup^n_{k=1} X_k = [0,1]$ and $m(X_k) = w_k$ for all $k$ ($m$ the Lebesgue measure).

Consider the surjection
\[
\psi :\{X \subseteq [0,1] \, \mid \, X \text{ Borel}, m(X) = \alpha\} \to \left\{(t_1, \ldots, t_n) \, \left| \, 0 \leq t_k \leq w_k, \sum^n_{k=1} t_k = \alpha \right. \right\}
\]
defined by
\[
\psi(X) = (m(X \cap X_1), \ldots, m(X \cap X_n)).
\]
If $X \subseteq [0,1]$ is Borel measurable with $m(X) = \alpha$, then
\[
\tau(N 1_X) = \int_X  \sum^n_{k=1} \lambda_k 1_{X_k}(s) \, ds = \sum^n_{k=1} \lambda_k t_k
\]
where $(t_1, \ldots, t_n) = \psi(X)$.
Since every $P \in \proj(L^\infty[0,1])$ is of the form $P = 1_X$ where $X \subseteq [0,1]$ and $\tau(P) = m(X)$, the result follows,
using Corollary \ref{cor:projection-nr}.
\end{proof}

For our next example, recall that a Haar unitary is a unitary element whose spectral distribution is Haar measure on the unit circle.
\begin{exam}
\label{exam:Haar}
Let $(\fM, \tau)$ be a tracial von Neuman algebra, let $U \in \fM$ be a Haar unitary, and let $\bD$ denote the closed unit disk.
For every $\lambda \in \bC$ with $|\lambda| = 1$, $\lambda U$ and $U$ have the same spectral distribution.
Therefore, Proposition \ref{prop:distributions} implies 
\[
V_\alpha(U) = V_\alpha(\lambda U) = \lambda V_\alpha(U)
\]
for every $\alpha \in (0,1]$ and $\lambda \in \bC$ with $|\lambda| = 1$.
Since each $V_\alpha(U)$ is a compact, convex set, this implies
\[
V_\alpha(U) = r(\alpha) \bD
\]
where $r : (0, 1] \to [0, 1]$ is such that $r(\alpha) = \sup\{\Re(z) \, \mid \, z \in V_\alpha(U)\}=\sup V_\alpha(\Re(U))$
where the last equality is part~\eqref{part:real} of Proposition~\ref{prop:basic-properties}.

To compute $r(\alpha)$, note  that by Proposition \ref{prop:distributions} we may assume that
$U = (s \mapsto e^{i s}) \in L^\infty[-\pi, \pi]$, so $\Re(U)=(s\mapsto\cos(s))$ and, arguing as in the proof of Proposition~\ref{prop:self-adjoints},
we deduce
\[
r(\alpha) =  \frac{1}{2\pi \alpha} \int^{\pi\alpha}_{-\pi \alpha} \cos(s) \, ds = \frac{1}{\pi \alpha} \sin(\pi \alpha).
\]
Thus $V_\alpha(U) = \frac{1}{\pi \alpha} \sin(\pi \alpha) \bD$ for all $\alpha \in (0, 1]$.
\end{exam}

The above example exhibits a method for computing $\alpha$-numerical ranges, provided there exists sufficient symmetry.

\begin{cor}
\label{cor:radial}
Let $(\fM, \tau)$ be a diffuse tracial von Neumann algebra and suppose $T \in \fM$ is such that 
\[
V_\alpha(T) = e^{i\theta} V_\alpha(T) \text{ for all } \theta \in [0, 2\pi).
\]
Then $V_\alpha(T)$ is the closed disk centered at the origin of radius $r_\alpha(T)$, where
\[
r_\alpha(T) = \frac{1}{\alpha} \int^\alpha_0 \lambda_{\Re(T)}(s) \, ds
=\sup V_\alpha(\Re(T)).
\]
\end{cor}

Of course, the above corollary applies whenever the $*$-distribution of $T$ is the same as the $*$-distribution of $e^{i\theta}T$
for all $\theta\in\bR$.

Using Method \ref{meth:g}, we may compute the $\alpha$-numerical ranges of several interesting operators.

\begin{exam}
\label{exam:Tucci}
Consider the infinite tensor view of the hyperfinite II$_1$ factor
\[
\fR = \bigotimes_{n\geq 1} \M_2(\bC)
\]
and consider the Tucci operator~\cite{Tu08}
\[
T = \sum_{n\geq 1} \frac{1}{2^n} (\underset{n-1\text{ times}}{\underbrace{I_2 \otimes \cdots \otimes I_2}} \otimes Q \otimes I_2 \otimes \cdots)
\]
where
$Q = \left[  \begin{smallmatrix} 0 & 1 \\ 0 & 0  \end{smallmatrix} \right]$.
This operator is quasinilpotent and generates $\fR$.
To compute $V_\alpha(T)$ for every $\alpha \in (0, 1]$, we first notice that $T$ and $e^{i\theta} T$ are approximately unitarily equivalent via the unitaries
\[
U_{n,\theta} = \left[  \begin{array}{cc} 1 & 0 \\ 0 & e^{-i\theta}  \end{array} \right] \otimes \left[  \begin{array}{cc} 1 & 0 \\ 0 & e^{-i\theta}  \end{array} \right] \otimes \cdots \otimes \left[  \begin{array}{cc} 1 & 0 \\ 0 & e^{-i\theta}  \end{array} \right] \otimes I_2 \otimes I_2 \otimes \cdots,
\]
as $U_{n,\theta}^* (e^{i\theta} T) U_{n,\theta}$ approximate $T$ in norm.  Therefore, Corollary \ref{cor:aue-implies-nr-equal} and Corollary \ref{cor:radial}  imply
\[
V_\alpha(T) = r_\alpha(T) \bD
\]
where $\bD$ denotes the closed unit disk and $r_\alpha(T)$ may be computed by as
\[
r_\alpha(T) = \sup( V_\alpha(\Re(T))).
\]

Let
\[
A_0 = \Re(Q) = \frac{1}{2}\left[  \begin{array}{cc} 0 & 1\\ 1 & 0  \end{array} \right].
\]
Then
\[
\Re(T) = \sum_{n\geq 1} \frac{1}{2^n}  (I_2 \otimes \cdots \otimes I_2 \otimes A_0 \otimes I_2 \otimes \cdots).
\]
However, since $2A_0$ is unitarily equivalent to 
\[
A =  \left[  \begin{array}{cc} 1 & 0\\ 0 & -1  \end{array} \right],
\]
we obtain that $\Re(T)$ is approximately unitarily equivalent to 
\[
S = \frac{1}{2} \sum_{n\geq 1} \frac{1}{2^n} (I_2 \otimes \cdots \otimes I_2 \otimes A \otimes I_2 \otimes \cdots).
\]
Thus, Corollary \ref{cor:aue-implies-nr-equal} implies 
\[
r_\alpha(T) = \sup( V_\alpha(S)).
\]

Notice
\[
\sum_{n=1}^2 \frac{1}{2^n} (I_2 \otimes \cdots \otimes I_2 \otimes A \otimes I_2 \otimes \cdots) = \mathrm{diag}\left(\frac{3}{4}, \frac{1}{4}, -\frac{1}{4}, -\frac{3}{4}  \right).
\]
Furthermore
\[ 
\sum_{n=1}^3 \frac{1}{2^n} (I_2 \otimes \cdots \otimes I_2 \otimes A \otimes I_2 \otimes \cdots) = \mathrm{diag}\left(\frac{7}{8}, \frac{5}{8}, \frac{3}{8}, \frac{1}{8}, -\frac{1}{8}, -\frac{3}{8}, -\frac{5}{8}, -\frac{7}{8}  \right).
\]
This pattern continues and thus  we see that the spectral scale of $S$ is
\[
\lambda_S(s) = \frac{1}{2}(1-2s).
\]
Thus,
\begin{align*}
r_\alpha(T) = \frac{1}{2\alpha} \int^\alpha_0 (1-2s) \, ds = \frac{1}{2}(1-\alpha)
\end{align*}
so
\[
V_\alpha(T) = \frac{1}{2}(1-\alpha) \bD.
\]
It is not very difficult to construct a normal operator $N$ satisfying $V_\alpha(N) = V_\alpha(T)$ for all $\alpha \in (0, 1]$, namely, having the same numerical ranges as the quasinilpotent operator $T$.
Indeed, considering the radially symmetric distribution $\nu$ on the unit disk such that $\nu(r\bD)=1-\sqrt{1-r^2}$ for $0<r<1$,
one can show that the marginal
distribution of $\nu$ is uniform measure on $[-1,1]$.
It follows that the normal operator $N$ whose trace of spectral measure is $\nu$ satisfies $\lambda_{\Re(N)}(s)=\frac12(1-2s)$ for all $s\in[0,1)$
and this implies $V_\alpha(N) = V_\alpha(T)$ for all $\alpha \in (0, 1]$.
\end{exam}

\begin{exam}
\label{exam:circular}
Recall a $(0,1)$-circular operator is an element $Z$ of a tracial von Neumann algebra of the form
\[
Z=\frac1{\sqrt 2}(X+iY),
\]
where $X$ and $Y$ are freely independent $(0,1)$-semicircular operators.
As the $*$-distribution of $Z$ is the same as the $*$-distribution of $e^{i\theta}Z$ for all $\theta\in\bR$,
Corollary \ref{cor:radial} implies 
\[
V_\alpha(Z) = r_\alpha(Z) \bD
\]
where $r_\alpha(Z) = \sup(V_\alpha(\Re(Z)))$.  Since the spectral distribution of $\Re(Z) = \frac{1}{\sqrt{2}}X$ is given by the semicircular law
\[
\frac{1}{\pi} 1_{\left[-\sqrt{2},\sqrt{2}\right]}(x)\sqrt{2-x^2},
\]
we obtain that
\[
r_\alpha(Z)=\frac{1}{\pi} \int_{h(\alpha)}^{\sqrt{2}}x\sqrt{2-x^2}\,dx=\frac{1}{3\pi\alpha}\left(2-h(\alpha)^2\right)^{3/2},
\]
where $h(\alpha) \in \left[-\sqrt{2}, \sqrt{2}\right)$ is such that
\[
\frac{1}{\pi}\int_{h(\alpha)}^{\sqrt{2}}\sqrt{2-x^2}\,dx=\alpha.
\]
Thus, 
$h$ is the inverse with respect to composition of the decreasing function $f: \left[-\sqrt{2},\sqrt{2}\right]\to[0,1]$ given by
\[
f(y)=\frac{1}{\pi}\int_y^{\sqrt{2}}\sqrt{2-x^2}\,dx=\frac{1}{2}-\frac{1}{2\pi}y\sqrt{2-y^2}-\frac1\pi\arcsin \left(\frac{y}{\sqrt{2}}\right).
\]
We note the asymptotic expansions
\begin{alignat*}{2}
f(\sqrt{2}-x)&=\frac{2^{7/4}}{3\pi}x^{3/2}-\frac1{5\pi2^{3/4}}x^{5/2}+O(x^{7/2})
&\quad&(\text{as }x\to0^+), \\[2ex]
h(\alpha)&=\sqrt{2}-\frac{(3\pi)^{2/3}}{2^{7/6}}\alpha^{2/3}-\frac{(3\pi)^{4/3}}{5(2^{23/6})}\alpha^{4/3}+O(\alpha^2)
&\quad&(\text{as }\alpha\to0^+), \\[2ex]
r_\alpha(Z)&=\sqrt{2}-\frac{3^{5/3}\pi^{2/3}}{5(2^{7/6})}\alpha^{2/3}+O(\alpha)
&\quad&(\text{as }\alpha\to0^+).
\end{alignat*}
For comparison, a $(0,1)$-circular element has norm $2$ and spectrum equal to the disk centred at the origin of radius $1$.
Note that, since the push-forward measure of the spectral distribution of the normalized Lebesgue measure on the disk of radius $\sqrt2$
onto the real axis produces the semicircular law $\frac1{\sqrt2}X$, $Z$ is an easy example of a non-normal operator such that there exists a normal operator $N$ with $V_\alpha(Z) = V_\alpha(N)$ for all $\alpha \in (0, 1]$.
\end{exam}

\begin{exam}
The quasinilpotent DT-operator $S$ was introduced in~\cite{DH04a} as one of an interesting class of operators in the free group factor $L(\bF_2)$,
that can be realized as limits of upper triangular random matrices.
As the name suggests, its spectrum is $\{0\}$, and it satisfies $\|S\|=\sqrt e$ and $\tau(S^*S)=1/2$.
In~\cite{DH04b}, it was shown that $S$ generates $L(\bF_2)$ and that $S$ has many non-trivial hyperinvariant subspaces.
Moreover, $\Re(S)=\frac12 X$, where $X$ is a $(0,1)$-semicircular operator and the $*$-distribution of $S$ is the same as that of $e^{i\theta}S$
for all $\theta\in\bR$.
Thus, the method of Corollary~\ref{cor:radial} applies, exactly as in Example~\ref{exam:circular}, to yield
\[
V_\alpha(S)=r_\alpha(S)\bD,
\]
where $r_\alpha(S)=\frac1{\sqrt2}r_\alpha(Z)$, where $r_\alpha(Z)$ is the function as defined in Example~\ref{exam:circular}.
Note that the normal measure whose distribution is uniform measure on the disk of radius $\frac1{\sqrt2}$ is has the same
$\alpha$-numerical ranges as the quasinilpotent operator $S$.
\end{exam}

\begin{exam}
As a generalization of Example \ref{exam:circular}, consider the operator
\[
T=\cos(\psi)X+i\sin(\psi)Y 
\]
where $\psi\in (0, \frac{\pi}{2})$ and $X$ and $Y$ are freely independent $(0,1)$-semicircular operators.
In particular, the case $\psi=\frac{\pi}{4}$ produces the circular operator studied in Example \ref{exam:circular}.
These elliptic variants of circular operators were studied by Larsen in \cite{L1999}, where he showed 
\begin{itemize}
\item $\left\|T\right\| = 2$, 
\item the spectrum of $T$ is $\left\{ z \in \bC \, \left| \, \frac{\Re (z)^2}{\cos^4(\psi)}+\frac{\Im (z)^2}{\sin^4(\psi)}\leq 4 \right.\right\}$, and
\item the Brown measure of $T$ is uniform distribution on its spectrum.
\end{itemize}

To determine $V_\alpha(T)$, we apply Method \ref{meth:g}.
Note that $\Re(e^{i\theta}T)$ is
\[
\cos(\psi)\cos(\theta)X-\sin(\psi)\sin(\theta)Y,
\]
which is $(0,b(\theta)^2)$-semicircular where
\[
b(\theta)= \sqrt{\cos^2(\psi)\cos^2(\theta)+\sin^2(\psi)\sin^2(\theta)}.
\]
Thus the spectral distribution of $\Re(e^{i\theta}T)$ is the same as the spectral distribution of
$\sqrt{2}\,b(\theta) \Re (Z)$, where $Z$ is the $(0,1)$-circular operator from Example \ref{exam:circular}.
Hence
\[
g_{V_\alpha(T)}(\theta)=\sqrt{2}\,r_\alpha(Z)b(\theta).
\]
Therefore, by Proposition \ref{prop:gK} and Example \ref{exam:ellipse-function},
we find
\[
V_\alpha(T)=\left\{z\in\bC\, \left| \,\frac{\Re (z)^2}{\cos^2(\psi)}+\frac{\Im (z)^2}{\sin^2(\psi)}\leq 2r_\alpha(Z)^2 \right. \right\}.
\]
It is curious, although not surprising, that
the eccentricity of the ellipse bounding $V_\alpha(T)$ is (except in the circular case $\psi=\frac{\pi}{4}$) different from the eccentricity
of the ellipse bounding the spectrum $\sigma(T)$.
\end{exam}

To complete this section, we note the following interpolation result that generalizes \cite{GS1976}*{Corollary 1}.
This enables one to obtain knowledge pertaining to one $\alpha$-numerical range based on others.
We note that further results in \cite{GS1976} also have immediate generalizations to $\alpha$-numerical ranges.
\begin{prop}
\label{prop:interpolation}
Let $(\fM, \tau)$ be a diffuse, tracial von Neumann algebra and let $T \in \fM$.  If $0 < \alpha < \beta < \gamma \leq 1$, then
\[
\frac{\alpha(\gamma - \beta)}{\beta(\gamma - \alpha)} V_\alpha(T) + \frac{\gamma(\beta  - \alpha)}{\beta(\gamma - \alpha)} V_\gamma(T) \subseteq V_\beta(T).
\]
\end{prop}
\begin{proof}
Let $\lambda \in V_\alpha(T)$ and let $\mu \in V_\gamma(T)$.
By definition, there exist positive contractions $X, Y \in \fM$ such that $\tau(X) = \alpha$, $\tau(Y) = \gamma$, 
\[
\lambda = \frac{1}{\alpha}\tau(TX), \qand \mu = \frac{1}{\gamma} \tau(TY).
\]
Let
\[
Z = \frac{\gamma - \beta}{\gamma - \alpha} X + \frac{\beta - \alpha}{\gamma- \alpha} Y \in \fM.
\]
It is clear that $Z$ is a positive operator such that
\[
Z \leq \frac{\gamma - \beta}{\gamma - \alpha} I_\fM + \frac{\beta - \alpha}{\gamma- \alpha} I_\fM = I_\fM
\]
and
\[
\tau(Z) = \frac{\gamma - \beta}{\gamma - \alpha} \alpha + \frac{\beta - \alpha}{\gamma- \alpha} \gamma =\beta.
\]
Finally,
\[
\frac{\alpha(\gamma - \beta)}{\beta(\gamma - \alpha)} \lambda + \frac{\gamma(\beta  - \alpha)}{\beta(\gamma - \alpha)}\mu = \frac{1}{\beta}\frac{\gamma - \beta}{\gamma - \alpha} \tau(TX) + \frac{1}{\beta} \frac{\beta - \alpha}{\gamma- \alpha} \tau(TY)
= \frac{1}{\beta} \tau(TZ) \in V_\beta(T),
\]
completing the proof.
\end{proof}

\begin{rem}
One may ask whether set equality must occur in Proposition \ref{prop:interpolation}.  Taking $T \in \fM$ to be a Haar unitary,
Example \ref{exam:Haar} implies this question asks (by letting $\gamma = 1$) whether
\[
\frac{1-\beta}{\pi(\beta - \alpha \beta)} \sin(\pi \alpha) \bD + 0 = \frac{1}{\pi \beta} \sin(\pi \beta)\bD
\]
holds for all $0 < \alpha < \beta < 1$.
As this is clearly not the case, equality need not occur in Proposition \ref{prop:interpolation}.  However, one may use \cite{AA2003-2} to demonstrate that equality occurs in Proposition \ref{prop:interpolation} when $T$ is an $n \times n$ matrix, $\alpha = \frac{k}{n}$, and $\gamma = \frac{k+1}{n}$ for some $k \in \{1, \ldots, n\}$.
\end{rem}

\section{Numerical Ranges and Diagonals}
\label{sec:Diagonals}

In this our final section, we desire description of when a scalar belongs to
the $\alpha$-numerical range of an operator based on the possible `diagonals' of an operator.  Our characterization is similar to that for $k$-numerical ranges of matrices found in \cite{FW1971}*{Theorem 2.4}.  Unfortunately, we do not obtain true `diagonals' as we do not know if one can guarantee $\A$ in the following technical lemma (whose proof is a generalization of a matricial result) is a MASA.
\begin{lem}
\label{lem:zero-diagonal}
Let $(\fM,\tau)$ be a type II$_1$ factor and let $T \in \fM$ be such that $\tau(T) = 0$.  Then there exists a diffuse abelian von Neumann subalgebra $\A$ of $\fM$ such that $E_\A(T) = 0$, where $E_\A : \fM \to \A$ is the normal conditional expectation.
\end{lem}
\begin{proof}
Notice $0 \in V_1(T) \subseteq V_{\frac{1}{2}}(T)$.  Hence there exists a projection $P \in \fM$ such that $\tau(P) = \frac{1}{2}$ and $\tau(TP) =0$.  Similarly, $\tau(T(I_\fM - P)) = 0$.  By repeating this argument in $P\fM P$ and $(I_\fM - P) \fM (I_\fM - P)$, we obtain four projections $\{P_k\}^4_{k=1}$ such that $P_k$ commutes with $P$ and $I_\fM - P$, $\tau(P_k) = \frac{1}{4}$, and $\tau(TP_k) = 0$ for all $k$.  By continuing to repeat the first argument on each compression and by taking the von Neumann algebra generated by these projections, the desired diffuse abelian von Neumann subalgebra of $\fM$ is obtained.
\end{proof}

\begin{prop}
Let $(\fM,\tau)$ be a type II$_1$ factor, let $T \in \fM$, and let $\alpha \in (0, 1]$.  Then $\lambda \in V_\alpha(T)$ if and only if there exists a diffuse abelian von Neumann subalgebra $\A$ of $\fM$ such that $\tau(1_{\{\lambda\}}(E_\A(T))) \geq \alpha$, where $E_\A : \fM \to \A$ is the normal conditional expectation.
\end{prop}
\begin{proof}
Suppose $\A$ a diffuse abelian von Neumann subalgebra of $\fM$ such that $\beta:=\tau(1_{\{\lambda\}}(E_\A(T))) \geq \alpha$.  Thus 
\[
\lambda = \tau(E_\A(T) 1_{\{\lambda\}}(E_\A(T))) = \tau(T 1_{\{\lambda\}}(E_\A(T))) \in  V_\beta(E_\A(T))\subseteq V_\alpha(E_\A(T)).
\]
(See Remark~\ref{rem:Valpha}.)

For the converse direction, suppose $\lambda \in V_\alpha(T)$.  By part (\ref{part:rotate-translate}) of Proposition \ref{prop:basic-properties}, we may  without loss of generality assume that $\lambda = 0$.  Since $0 \in V_\alpha(T)$, by Corollary \ref{cor:projection-nr} there exists a projection $P$ of trace $\alpha$ such that $\frac{1}{\alpha} \tau(TP) = 0$.  Hence $\tau_{P\fM P}(PTP) = 0$ where $\tau_{P\fM P}$ is the trace for $P\fM P$.  By Lemma \ref{lem:zero-diagonal} there exists a diffuse abelian von Neumann subalgebra $\A_0$ of $P\fM P$ such that $E_{\A_0}(PTP) = 0$.  If $\A'$ is any diffuse abelian von Neumann subalgebra of $(I_\fM - P) \fM (I_\fM - P)$, then $\A = \A_0 \oplus \A' \subseteq \fM$ is a diffuse abelian von Neumann subalgbebra containing $P$ such that $E_\A(T)P = 0$. Hence $\tau(1_{\{\lambda\}}(E_\A(T))) \geq \alpha$ as desired.
\end{proof}

\end{document}